\documentclass[psamsfonts, reqno]{amsart}
\usepackage{amssymb,amsfonts}
\usepackage[all,arc]{xy}
\usepackage{enumerate}
\usepackage{mathrsfs}
\usepackage{graphicx}\usepackage[margin=0.5in]{geometry}

\usepackage{caption}
\numberwithin{figure}{section}
\usepackage{tikz-cd}
\newtheorem{thm}{Theorem}[section]
\newtheorem{cor}[thm]{Corollary}
\newtheorem{prop}[thm]{Proposition}

\newtheorem{lem}[thm]{Lemma}
\newtheorem{conj}[thm]{Conjecture}

\theoremstyle{definition}
\newtheorem{defn}[thm]{Definition}

\newtheorem{exmp}[thm]{Example}

\newtheorem{notn}[thm]{Notation}

\theoremstyle{remark}
\newtheorem{rem}[thm]{Remark}

\newcommand{\eqal}{\begin{equation}\begin{aligned}}
\newcommand{\eqalnt}{\begin{equation}\begin{aligned}\notag}
\newcommand{\eeqal}{\end{aligned}\end{equation}}

\newcommand{\fm}{{ s }}
\newcommand{\fN}{L}

\newcommand{\nm}{{\mathcal{N}}}

\newcommand{\NN}{{\mathbb N}}
\newcommand{\QQ}{{\mathbb Q}}

\newcommand{\ZZ}{{\mathbb Z}}

\newcommand{\sa}{{\mathfrak{P}}}
\newcommand{\sO}{{\mathcal{O}}}

\numberwithin{equation}{section}

\title{Counting extensions of number fields with Frobenius Galois group}

\author{Harsh Mehta}

\begin{document}

\begin{abstract}

Let $G$ be a Frobenius group with an abelian Frobenius kernel $F$ and let $k$ be a finite extension of $\mathbb{Q}$. 
We obtain an upper bound for the number of degree $|F|$ algebraic extensions $K/k$ with Galois group $G$ with the norm of the discriminant $\mathcal{N}_{k/\mathbb{Q}}(d_{K/k})$ bounded above by $X$. We extend this method for any group $G$ that has an abelian normal subgroup. If $G$ has an abelian normal subgroup, then we obtain upper bounds for the number of degree $|G|$ extensions $N/k$ with Galois group $G$ with bounded norm of the discriminant. Malle made a conjecture about what the order of magnitude of this quantity should be as the degree of the extension $d$ and underlying Galois group $G$ vary. We show that under the $\ell$-torsion conjecture, the upper bounds we achieve for certain pairs $d$ and $G$ agree with the prediction of Malle. Unconditionally we show that the upper bound for the number of degree 6 extensions with Galois group $A_4$ also satisfies Malle's weak conjecture.  
\end{abstract}

\maketitle

\subsection{This is a preliminary version of this work}
 \section{Introduction}

Let $K/k$ be a degree $d$ extension that lies in a fixed algebraic closure of $\QQ$. 
 We assume that $G=\text{Gal}(\hat{K}/{k})$ is non-trivial and view it as a permutation group $G\leq S_d$ that acts transitively on $d$ letters. We are interested in $N_d(k,G;X)$ as $X\rightarrow\infty$, where
 \eqal\label{goal}
 N_d(k,G;X)&=|\{ K/k : \mbox{Gal}(\hat{K}/k)\cong G, [K:k]=d \mbox{, and } \nm_{k/\QQ}(d_{K/k})\leq X\}|.\\
 \eeqal 
Malle made a conjecture \cite{Mal04} on what the order of magnitude of $N_d(k,G;X)$ should be, and in order to
state it we need notation. 

\begin{defn}\label{agd} Let $G$ be a non-trivial subgroup of the permutation group $S_d$. Let $G$ act transitively on $[d]:=\{1,2,\dots,d\}$ and let $g\in G$.   

1. The index of $g$, is $\mbox{ind}(g):=d-$ the number of orbits of $g$ on $[d].$

2. $\mbox{ind}(G):=\min\{\mbox{ind}(g): 1\neq g\in G\}.$

3. $a(G,d):=1/\mbox{ind}(G).$
\end{defn}
Now we are in a position to state what Malle conjectured. Malle conjectured that 
\[N_d(k,G;X)\sim c(k,G)X^{a(G,d)}\log^{b(k,G,d)-1}(X), \]
for a certain explicit function $b$ that depends on $k$, $G$ and $d.$ For a precise formulation of his original conjecture see \cite{Mal04}. For our purpose it suffices to consider a weaker form of the conjecture. 
\begin{conj}\label{Malle}\textnormal{[Malle's weak conjecture]}
For any non-trivial group $G\leq S_d$ acting transitively on $[d]$, and any number field $k$,
\eqal\label{M1}
 X^{a(G,d)}\ll N_d(k,G;X)\ll X^{a(G,d)+\epsilon}\\
 \eeqal
 holds for all $\epsilon>0$ as $X\rightarrow\infty.$
\end{conj}

Throughout this work, we say $f(X)\ll g(X)$ when there exist positive constants $C$ and $N$ such that for all $X>N$, we have 
$|f(X)|\leq C|g(X)|.$ We say $f(X)=O(g(X))$ if and only if $f(X)\ll g(X)$.

The purpose of this work is to establish upper bounds for $N_d(k,G;X)$ for certain pairs $(G,d).$ We study the case that $G$ is a finite non-trivial subgroup of $S_d$ as above, and impose another condition on $G$: $G$ must satisfy
 \eqal\label{e-seq}
 0\rightarrow F\rightarrow G\rightarrow H\rightarrow 0
 \eeqal
where $F$ is non-trivial and abelian. With $G$ described as above, we develop the work of Kl{\"u}ners \cite{JK1} and Ellenberg and Venkatesh \cite{EV1} to obtain upper bounds for $N_{|G|}(k,G;X).$ Moreover, if $G$ is a Frobenius group with an abelian Frobenius kernel $F,$ we make use of a Brauer relation to obtain upper bounds for $N_{|F|}(k,G;X).$

\begin{defn}
Let $\mathcal{F}_1$ be the set of groups $\{1\}\neq G\leq S_d$ that act transitively on $[d]$ and satisfy \eqref{e-seq}. Let $F$ and $H$ be defined as in \eqref{e-seq}. We define $\mathcal{F}$ to be the subset of $\mathcal{F}_1$ that contains all groups $G$ that are Frobenius, with Frobenius kernel $F.$ $G$ is said to be Frobenius if for all $g\in G\setminus H$, $H\cap H^g=\{1\}$ where $H^g:=\{ghg^{-1}: h\in H\}.$

\end{defn}

\begin{exmp}
For odd $m$, $\mathcal{F}$ contains dihedral groups $D_{m}$. Let $C_n$ denote the cyclic group of order $n$. $\mathcal{F}$ also contains groups of the form $C_{\ell}\rtimes C_{\ell-1}$, where $\ell$ is an odd prime. Groups such as $A_4\times C_2$ are contained in $\mathcal{F}_1\setminus \mathcal{F}.$
\end{exmp}
Before stating the main result, we give an example of what the expected order of magnitude for $N_d(k,G;X)$ is under Malle's conjecture.

\begin{exmp}
Let $G=D_{\ell}=\{r,s | r^{\ell}=s^2=(sr)^2=1\}$ be the dihedral group of size $2\ell$, with $\ell$ an odd prime.
Let $G$ act on $[\ell]=\{1,\dots,\ell\}$ by the induced left multiplication action of $S_{\ell}$ acting on $[\ell]$. In particular, 

\[s=(1)(\ell\;\;  2)(\ell-1\;\;3)(\ell-2\;\;4)\dots\left(\frac{\ell+3}{2}\;\; \frac{\ell+1}{2}\right)\qquad \qquad r=(1\; 2\; \dots\; \ell).\]
The rotation $r$ has one orbit, therefore $\mbox{ind}(r)=\ell-1$. 
The reflection $s$ has one orbit of length one (the fixed point) and $(\ell-1)/2$ orbits of length 2, implying there are 
$(\ell+1)/2$ orbits in total. This implies that elements of order 2, have the smallest index, $\text{ind}(s)=(\ell-1)/2.$ Thus $a(D_{\ell},\ell)=2/(\ell-1)$. By a similar argument (provided explicitly in Section 2) one can conclude that $a(D_{\ell},2\ell)=1/\ell.$
\end{exmp}

\begin{notn}\label{H-int}
Throughout this work $\epsilon$ will be an arbitrary (small) positive constant.
With the notation in Definition \eqref{e-seq}, all groups 
in $\mathcal{F}_1$ are finite, with $|F|=m$ and $|H|=t$. Let $p$ and $p_1$ denote the smallest prime divisors of $m$ and $t$ respectively. 
Let $M/k$ be a Galois extension with Galois group $H$. Let
 $\mbox {Cl}_M[m]$ be the $m$-torsion elements of the ideal class group of $M$.
 Let $\mathcal{D}(H,m)=\mathcal{D}$ be any known constant such that 
$|\mbox{Cl}_M[m]|\ll d_{M/\QQ}^{\mathcal{D}}.$ Let $a_1(G,d)$ denote a known constant such that \[N_d(k,G;X)\ll X^{a_1(G,d)+\epsilon}.\] Let $a(G,d)$ denote the constant defined in Definition \ref{agd}.
Corresponding to this notation, we have the following field diagram.
\begin{equation}\label{d1}
  \begin{tikzcd}
    &N   &  \\ 
       &   &  \\
    K=\mbox{Fix}(H)\arrow{uur} {|H|=t}&  &   \\
           &   &M =\mbox{Fix}(F)  \arrow{uuul}{|F|}  \\
       &   &  \\
& k\arrow{uuul}{|F|=m} \arrow[swap]{uur}{|H|}\arrow[swap]{uuuuu}{|G|}  &  
  \end{tikzcd}
\end{equation}
\end{notn}

The main result of this paper is the following:
 \begin{thm}\label{main-2}
With notation as above, we have
 \[N_d(k,G;X)\ll X^{A(G,d)+\epsilon}\]
where $A(G,d)$, $d$, and $G$ are given by: 
 
 \begin{center}
\begin{tabular}{|l|l|l|}
\hline
$G$ &  $d$ & $A(G,d)$\\
\hline
$G\in \mathcal{F}$ & $m$ & $\max\left(\frac{(\mathcal{D}+a_1(H,t))\times t}{m-1},\frac{p}{m(p-1)}\right)$ \\
$G\in\mathcal{F}_1$ & $mt$ & $\max\left(\frac{a_1(H,t)+\mathcal{D}}{m},\frac{p}{mt(p-1)}\right)$\\
\hline
\end{tabular}\end{center}
Here, \[\mathcal{D}=\min_{\substack{p\\ p|m}}(\mathcal{D}(k,H,p)).\] 

\end{thm}
If we are able to attain better upper bounds for $\mathcal{D}$ then $A(G,d)$ may reduce, implying a tighter upper bound. It is believed that 
$\mathcal{D}$ can be arbitrarily small.
\begin{conj}{($\ell$-torsion conjecture)}
Let $K/\QQ$ be a number field of degree $n$. Then for every $\ell\in\NN,$
$
|\mbox{Cl}_K[\ell]|\ll_{n,\ell,\epsilon}d_{K/\QQ}^{\epsilon}.
$
\end{conj}
The impetus for this conjecture may be found in Duke \cite{D98}, Zhang \cite{Z05} and Brumer and Silverman \cite{BSS96}.

\begin{cor}\label{main-1}
Let $G$ be a group in $\mathcal{F}_1.$ If we assume the following conditions:

\begin{itemize}
    \item The $\ell$-torsion conjecture holds, in particular, we need that $|\mbox{Cl}_K[\ell]|\ll_{K,\ell,\epsilon}d_{K/\QQ}^{\epsilon}$ for any prime divisor $\ell$ of $[N:M].$
    
    \item $N_{|H|}(k,H;X)$ satisfies Malle's conjecture.
\end{itemize}

then the following hold true
\begin{enumerate}
    \item For any $G$ in $\mathcal{F}_1$, $N_{|G|}(k,G;X)$ satisfies Malle's conjecture.
    
    \item For any $C_m\rtimes C_t$ in $\mathcal{F}$, $N_{|F|}(k,G;X)$ satisfies Malle's conjecture.
\end{enumerate}

\end{cor}

\begin{exmp}\label{exmp1}
With the same notation as in Theorem \ref{main-2} above we have the following results:
\begin{center}
\begin{tabular}{|l|l|l|l|l|}
\hline
$G$ & Conditions & $d$ & $A(G,d)$ & $a(G,d)$\\
\hline
$C_{\ell}\rtimes C_{\ell-1}$& $\ell$ is an odd prime & $\ell$ & $\frac{1}{2}+\frac{2}{\ell-1}$&${2}/{\ell-1}$\\
$C_{\ell}\rtimes C_{\ell-1}$&  $\ell$ is an odd prime & $\ell^2-\ell$ & $\frac{1}{2\ell}+\frac{2}{\ell(\ell-1)}$&${2}/{(\ell(\ell-1))}$\\
$A_4$&  & $4$ & $0.7783$ & 1/2\\
$C_2^3\rtimes (C_7\rtimes C_3)$ & & $168$ & $9/112$& 1/84\\
$C_2^3\rtimes (C_7\rtimes C_3)$ & $\ell$-torsion conjecture & $168$ & $1/84$ & 1/84\\
$C_{2}^3\rtimes C_{7}$& $k=\QQ$  & $8$ & $0.595$ & 1/4\\
$C_{103}\rtimes C_{17}$ & $k=\QQ$ & 103 &0.0104 & 0.09369\\

$C_3^2\rtimes C_4$ & &$6$ & $1/2$&1/2\\
$A_4$ & &$6$ & $1/2$&1/2\\

\hline
\end{tabular}\end{center}
\end{exmp}

The first 3 examples are direct consequences of Theorem \ref{main-2}, the next two examples are shown by applying Theorem \ref{main-2} twice. The last four examples are extensions of the method used to prove Theorem \ref{main-2}. The last two examples indicate that the method can be used to obtain upper bounds for $N_d(k,G;X)$ for $d\neq mt$ in certain cases. In fact, the last examples implies that, unconditionally, we have $N_6(k,A_4;X)\ll X^{1/2+\epsilon}$ and $N_6(k,C_3^2\rtimes C_4;X)\ll X^{1/2+\epsilon}$ which are exactly as Malle's conjecture predicts. Details regarding these extensions may be found in Section \ref{exts}.

We outline the main ideas behind obtaining an upper bound for $N_{mt}(k,G;X)$ for $G\in\mathcal{F}_1$. Fix the notation as in Diagram \ref{d1}.

\begin{itemize}

\item Fix a base field $k$ and a degree $t$ extension $M/k$ with Galois group $H.$
The idea is to exploit the discriminant relation $d_{N/k}=d_{M/k}^{m}N_{M/k}(d_{N/M}).$

\item  We count the number of abelian extensions $N/M$ of degree $m$ with the discriminant $d_{N/M}$ having a certain fixed support. To count $N_{mt}(k,G;X)$, we sum over all $M/k$ and the corresponding possible supports of $d_{N/M}.$ Precisely, $N_{mt}(k,G;X)$ is bounded above by
\[|\{M/k \text{ with } \nm_{k/\QQ}(d_{M/k})\leq X^{1/m}\}|\times\Big|\Big\{\text{abelian extensions  }  N/M \text{ of degree }m, \text{ with }\nm_{M/\QQ}(d_{N/M})\leq \frac{X}{\nm_{k/\QQ}(d_{M/k})^{m}}\Big\}\Big|.\]

\item Show that the number of distinct integer values $\nm_{M/\QQ}(d_{N/M})\leq X$ above can take is $O(X^{1/R(G)})$ for some function $R$ that depends only on $G$. We can do this by finding the power of the primes $\ell\in\ZZ$ that ramify tamely in $N$ and $\ell\nmid \nm_{k/\QQ}(d_{M/k})$. This is made precise in Lemma \ref{powerlem}.

\end{itemize}

Obtaining upper bounds for $N_m(k,G;X)$ for groups in $\mathcal{F}$ is done in the same manner as above, with a different discriminant relation, one stemming from a Brauer relation. 

 \subsection{Known results}
Malle made his conjecture in 2002, though the problem of counting number fields ordered by an invariant has been investigated earlier. A conjecture in this direction, generally attributed to Linnik, states that if we fix a number field $k$ the number of degree $d$ extensions $K/k$ with $\nm_{k/\QQ}(d_{K/k})\leq X$ for $X\rightarrow\infty$ is $O(X).$ We denote the statement of Linnik's conjecture as
 \eqal\label{basic}
 N_d(k;X)=O_{d,k}(X).\eeqal
 The conjecture holds for $d\leq 5.$. The cases of $d=1$ and $N_2(\QQ;X)$ are trivial and do agree with the conjecture. The order of magnitude for $N_2(k;X)$ and $N_3(k;X)$ for general $k$ were
 obtained by Datskovsky and Wright \cite{DD88} who show that 
\[N_3(k;X)= C_kX +o(X)\]
 for an explicit constant $C_k$ that depends on $k.$ 
Manjul Bhargava showed that $N_4(\QQ;X)\sim c_4X$ and $N_5(\QQ;X)\sim  c_5X$ in \cite{MB05} and \cite{MB10}. For arbitrary $k$ and permutation groups $S_d$,  Bhargava, Shankar and Wang \cite{bsw} compute explicit constants $c_d$ in  $N_d(k,S_d;X)=c_dX+o(X)$ for $d=2,3,4$ and $5.$ 
Ellenberg and Venkatesh \cite{EV1} establish upper bounds for all $d>3$,  
precisely, they show for a positive constant $C$,
\[N_d(k;X)\ll X^{\mbox{exp}(C\sqrt{\log{d}})}.\]
The upper bounds that are conjectured by Malle are never greater than  one, thereby making it a stronger conjecture than \eqref{basic}. 
For abelian groups, Wright \cite{DW1} established the order of magnitude of $N_{|G|}(k,G;X)$, showing that it does satisfy Malle's conjecture. For any group $G$ with an $|G|>4$, Ellenberg and Venkatesh \cite{EV1} show that $N_{|G|}(k,G;x)\ll X^{3/8+\epsilon}$. The result in Theorem \ref{main-2} match the results in \cite{EV1} and only do better when there is non trivial information about the size of the torsion of the class group available to use. The method of proof for bounding $N_{|G|}(k,G;X)$ for $G\in\mathcal{F}_1$ in \cite{EV1} is almost the same as the method here except they do not make explicit use of the size of the class group.       
Kl{\"u}ners was the first to explicitly study a Frobenius group, $G=D_{\ell}$ where $\ell$ is an odd prime. We generalize his techniques to all groups in $\mathcal{F}.$
Kl{\"u}ners showed the following 
\eqal\label{jk1}
N_{\ell}(k,D_{\ell};X)\ll X^{\frac{3}{\ell-1}+\epsilon}\qquad\qquad N_{2\ell}(k,D_{\ell};X)\ll X^{\frac{3}{2\ell}+\epsilon}.
\eeqal
The first of the above results was improved upon by Cohen and Thorne in \cite{CT16}, Theorem 1.1, in the case that $k=\QQ$. Their methods also imply an improvement in $N_{2\ell}(\QQ,D_{\ell};X)$. Their method applies non-trivial bounds for the average value of $\text{Cl}_{\QQ(\sqrt{D})}[\ell],$ they show
\eqal\label{CT-dl}
N_{\ell}(\QQ,D_{\ell};X)\ll X^{\frac{3}{\ell-1}-\frac{1}{\ell^2-\ell}+\epsilon}.\eeqal
The above result is implied by work of Pierce, Turnage-Butterbaugh and Wood \cite{PTBW}. Pierce, Turnage-Butterbaugh and Wood additionally establish non-trivial upper bounds for the size of the $m$-torsion of the class group for almost all Galois extensions $M/\QQ$ with a wide range of possible Galois groups. Using these bounds for quadratic extensions, Theorem \ref{main-2} subsumes the result of Cohen and Thorne for $N_{\ell}(\QQ,D_{\ell};X)$ and improves on the result of Kl{\"u}ners' for $N_{2\ell}(\QQ,D_{\ell};X)$ by showing
\[N_{2\ell}(\QQ,D_{\ell};X)\ll X^{\frac{3}{2\ell}-\frac{3}{2\ell^2}+\epsilon}.\]
For the Frobenius group $C_5\rtimes C_4,$ we are able to show $N_5(k,C_5\rtimes C_4;X)\ll X^{1+\epsilon}$ which is not as tight as the upper bound of Bhargava, Cojocaru and Thorne, \cite{BCT}
who show 
\eqal\label{bct}
N_{5}(k,C_{5}\rtimes C_4;X)\ll X^{39/40+\epsilon}.\eeqal
Alberts \cite{BA} showed upper bounds for $N_d(k,G;X)$ in the case that $G$ is a solvable group. The set of solvable groups is a subset of the groups in $\mathcal{F}_1.$ Upper bounds for $N_{|G|}(k,G;X)$ for solvable groups that we are able to show in this paper are as tight, or tighter than the upper bounds for the same quantity in Alberts' work. For Frobenius groups $G\in\mathcal{F}$, upper bounds for $N_{|F|}(k,G;X)$ are tighter in this work than they are in \cite{BA}. Alberts' method has other advantages such as it may be used to count number fields ordered by an invariant other than the discriminant, such as the conductor. Alberts is also able to obtain upper bounds for $N_d(k,G;X)$ for all $d\neq |G|$.
Recently, Alberts \cite{BA2} wrote another paper that computed lower bounds for $N_d(k,G;X)$ as well for $G\in\mathcal{F}_1.$



This paper is structured as follows. Section 2 discuss the structure of Frobenius groups and introduces the Brauer relation. Section 3 contains technical lemmas needed to prove the main theorem, and then proves the main theorem. Section 4 is divided into two sections that talk about different ways to generalize the method of the proof. Section 4.1 explores how the method may be extended to obtain upper bounds for $N_d(k,G;X)$ for $d\notin\{m,mt\}$, for instance, we show $N_6(k,A_4;X)\ll X^{1/2+\epsilon}.$ Section 4.2 explores how non-trivial bounds for $|\text{Cl}_M[\ell]|$ for almost all extensions $M/k$  for a certain family of field extensions may be used to improve $N_d(k,G;X)$.
And finally, Section 5 inspects in which situation a stronger bound for $|\text{Cl}_M[\ell]|$ does and does not play a further role in improving upper bounds for $N_d(k,G;X).$

%
%
\section{Frobenius groups}\label{21}
 We introduce Frobenius groups here and compute the expected order of magnitude for $N_m(k,C_m\rtimes C_t;X)$. A Frobenius group $G\leq S_m$ acts transitively on $[m]$ such that no non-trivial element fixes more than one point and some non-trivial element fixes a point.
Frobenius groups have form $G=F\rtimes  H$ where the normal subgroup $F$ is referred to as the Frobenius kernel and $|H|=t$ is a divisor of $m-1.$ We now explore the structure of a certain family of Frobenius groups, those that have the form $G=C_m\rtimes C_t.$ Under Notation \ref{H-int}
we compute $a(G,m)$ and $a(G,mt).$ We represent $G$ as

\[C_m\rtimes C_t=\{\psi,\sigma: \psi^t=\sigma^m=1, \psi\sigma\psi^{-1}=\sigma^v\}\]
where $v$ is any primitive root modulo $m.$ The action of $G$ on a set of $m$ elements is unique. When $G$ acts on the set of $m$ elements, every $\sigma$ has orbits of length $s$ and a total of $m/s$ orbits, where $s$ is some divisor of $m.$ Every $\psi$ has one fixed point in this action, and every other orbit has length $j$ where $j$ is some divisor of $t.$   
Therefore 
\[\mbox{ ind}(\sigma)=m-\frac{m}{s}\qquad \qquad \mbox{ ind}(\psi)=m-\left(1+\frac{m-1}{j}\right).\] This implies 

\eqal\label{a(G)=m}
a(C_m\rtimes C_t,m)=\frac{1}{m-\max\left(\frac{m}{p},1+\frac{m-1}{p_1}\right)}.\\
\eeqal
We now compute $a(G,mt)$. More generally, we show that for any group $G$, if $p$ is the smallest prime divisor of $|G|$, then   
\eqal\label{a(G)=mt}
 a(G,|G|)=\frac{1}{|G|}\frac{p}{p-1}.\\
 \eeqal
 Every group $G$ has an element $g$ of order $p$. Every orbit of $g$ has length $p$ and hence $g$ has $|G|/p$ orbits. No non-identity element can have more orbits since the length of each orbit must be a divisor of $|G|.$ Hence $g$ is the element with most orbits, which shows the claim.
 
The main tool in showing upper bounds for $N_m(k,G;X)$ when $G$ is Frobenius, is a Brauer relation. By a result of Kl{\"u}ners and Fieker (\cite{FK1}, Theorem 4):
\begin{thm}\label{th-kl-fekir}
Fix an algebraic number field $k$. Let $G=F\rtimes H$ be any Frobenius group. Let $N/k$ be a normal extension with
 $\mbox{Gal}(N/k)= G$. Let $K$ be the fixed field of $H$ and $M$ be the fixed field of $F$. Then 

\begin{equation}\label{K/k-rel}
d_{K/k}=d_{M/k}^{(|F|-1)/|H|}\nm_{M/k}(d_{N/M})^{1/|H|}.
\end{equation}

\end{thm}
Observe that $F$ need not be abelian.

%
%
\section{Upper bounds}
This section is divided into two subsections, the first establishes the technical tools needed to prove Theorem \ref{main-2}, and  the second subsection proves it. A broad overview of the procedure can be found in the introduction. Throughout this section, we use the notation established in Notation \ref{H-int} and Diagram \ref{d1}.

\subsection{Technical results}

%
%

We now make more precise the definition of $\mathcal{D}$ that was defined in Notation \ref{H-int}. Let $M/k$ be a Galois extension with Galois group $H$. Then
\eqal\label{H-def}\mathcal{D}=
\mathcal{D}(k,H,m):=\limsup_{d_{M/\QQ}}\frac{\log(|\mbox{Cl}_M[m]|)}{\log(|d_{M/\QQ})|}.\\
\eeqal

\begin{lem}\label{cftlem}

Let $M/k$ be a finite extension of degree $t.$ Let $P$ be a finite set of primes in $\sO_M.$ The number of abelian extensions $N/M$ of degree $m$ which are at most ramified in $P$ is bounded above by
\[C^{|P|}O_{k}\left(\nm_{k/\QQ}(d_{M/k})^{\mathcal{D}}\right).\]
\end{lem}

\begin{proof}

Let $\mathfrak{m}=\mathfrak{m}_0\mathfrak{m}_{\infty}$ such that  $\mathfrak{m}_0$ is the product of primes in $P$  and $\mathfrak{m}_{\infty}$ consists of the real places of $M.$ Every extension $N/M$ as described above is a subfield of the ray class field of $\mathfrak{m}.$

Denote the ray class group of $\mathfrak{m}$ by $\text{Cl}_M(\mathfrak{m})$. The following is an exact sequence for ray class groups,
\[\mathcal{O}^{\times}_M\rightarrow(\mathcal{O}_M/\mathfrak{m})^{\times}\rightarrow \text{Cl}_M(\mathfrak{m})\rightarrow \text{Cl}_M\rightarrow 1. \]
By class field theory, subgroups of index $m$ of $\text{Cl}_M(\mathfrak{m})$ are in bijection with abelian extensions $N/M$ of degree $m$ with ramified primes contained in $P.$ Thus we compute the $m$-torsion of the ray class group $\text{Cl}_M(\mathfrak{m})$.
By the exact sequence above, 
\[|\text{Cl}_M(\mathfrak{m})[m]|\leq|(\mathcal{O}_M/\mathfrak{m})^{\times}[m]|\times |\mbox{Cl}_M[m]|.\]
By the Chinese remainder theorem, with $M_{\nu}$ denoting the completion of $M$ with respect to the norm $\nu$, we have 
\[(\sO_M/\mathfrak{m})^{\times}=(\sO_M/\mathfrak{m}_0)^{\times}\oplus(\sO_M/\mathfrak{m}_{\infty})^{\times}=\prod_{\mathfrak{p}\in P}(\sO_M/\mathfrak{p}^{e_{\mathfrak{p}}})^{\times}\prod_{\nu|\mathfrak{m}_{\infty}}M_{\nu}^{\times}/M_{\nu}^{+}.\]
Hence we have that $|(\sO_M/\mathfrak{m}_0)^{\times}[m]|$  
is bounded above by $2(m^{[M:\QQ]|P|}).$
There are at most $[M:\QQ]$ distinct real places, hence the contribution to the $|(\sO/\mathfrak{m})^{\times}[m]|$ from the real places is bounded above by $2^{[M:\QQ]}.$
Hence, by choosing $C\geq (2m)^{2[M:\QQ]}$,
\[|(\mathcal{O}_M/\mathfrak{m})^{\times}[m]|\leq C^{|P|}.\]By \eqref{H-def},
\eqalnt
& |\mbox{Cl}_M[m]|\ll d_{M/\QQ}^{\mathcal{D}}\ll (d_{k/\QQ}^t\nm_{k/\QQ}(d_{M/k}))^{\mathcal{D}}\ll_{k}\nm_{k/\QQ}(d_{M/k})^{\mathcal{D}}.\\\eeqal
By combining the two pieces above, we have that the number of abelian extensions of degree $m$ with discriminant supported on $\mathfrak{m}$ is bounded above by $C^{|P|}O_{k}(\nm_{k/\QQ}(d_{M/k})^{\mathcal{D}}).$

\end{proof}

\begin{lem}\label{powerlem}
Fix a tower of number fields $M/k/\QQ$. Fix $G\in\mathcal{F}_1$ and its corresponding normal abelian subgroup $F.$ For any integer $s$ such that $s|m$ and $s>1$ we have that

\begin{equation}\label{power}
\Bigg|\Bigg\{\begin{aligned}\fN/M:& [\fN:M]=\fm \text{, } \fN/M\text{ is abelian}\\ \text{Gal}(\hat{\fN}/M)=F\text{, } &\text{Gal}(\hat{\fN}/k)=G,\text{ } \nm_{M/\QQ}(d_{\fN/M})\leq X \end{aligned}\Bigg\}\Bigg|\ll_{G,\epsilon,k}C^{\omega(\nm_{k/\QQ}(d_{M/k}))} \nm_{k/\QQ}(d_{M/k})^{\mathcal{D}+\epsilon} X^{1/R}
\end{equation}
where $R$ is the largest integer such that for all primes $q\in\ZZ$ such that $q\nmid d_{M/\QQ}[\fN:\QQ]$ and $q|\nm_{M/\QQ}(d_{\fN/M})$, $q^R|\nm_{M/\QQ}(d_{\fN/M}).$

If $\fN/k$ is a Galois extension, $R\geq |G|(1-p^{-1})$, if $\fN/M$ is not Galois, $R \geq p-1$ where $p$ is the smallest divisor of $|F|$.

\end{lem}

The tower of fields $N/L/M/k$ may be represented as follows:
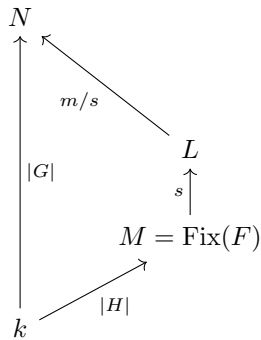
\begin{equation}\notag
  \begin{tikzcd}
    N   &  \\ 
        &   \\
    &   L\arrow{uul}{m/\fm}   \\
        & M =\mbox{Fix}(F)\arrow{u}{\fm} \\
 k\arrow[swap]{ur}{|H|}\arrow[swap]{uuuu}{|G|}  &  \\
  \end{tikzcd}
\end{equation}
\begin{proof}

Let $P_{L/M}(\sO_M)$ denote the set of primes in $M$ that ramify in the extension $L/M.$
Let $P_{\fN/M}(\ZZ)$ be the set of primes in $\ZZ$ that divide $\nm_{M/\QQ}(d_{\fN/M})$ and let $\ell$ denote a prime in $\ZZ$. Since $|P_{\fN/M}(\ZZ)|=O_{k,G}(|P_{\fN/M}(\sO_M)|)$, by Lemma \ref{cftlem}, with $M/\QQ$ and $P_{\fN/M}(\ZZ)$ fixed, we have
\[\sum_{\substack{\fN/M\\ [\fN:M]=\fm\text{, }\fN/M\text{ is abelian}\\ \ell|\nm_{M/k}(d_{\fN/M})\rightarrow \ell\in P_{L/M}(\ZZ)}}1\ll_{G,\epsilon,k} \nm_{k/\QQ}(d_{M/k})^{\mathcal{D}} C^{|P_{\fN/M}(\ZZ)|}.\] 
Now we relax the condition that $P_{\fN/M}(\ZZ)$ is fixed, and sum over all degree $\fm$ extensions $\nm_{M/\QQ}(d_{\fN/M})\leq X$. First we define $\mathbb{P}_M(X)$ as

\[\mathbb{P}_M(X)=\bigcup_{\substack{\fN/M\\ [\fN:M]=\fm, \fN/M \text{ is abelian}\\\nm_{M/\QQ}(d_{\fN/M})\leq X}}P_{\fN/M}(\ZZ).\]

 Hence we have

\eqal\label{power1}
\sum_{\substack{\fN/M\\ [\fN:M]=\fm\text{, } \fN/M\text{ is abelian}\\ \text{Gal}(\hat{\fN}/M)=F\text{, } \text{Gal}(\hat{\fN}/k)=G\\  \nm_{M/\QQ}(d_{\fN/M})\leq X }}1\ll \nm_{k/\QQ}(d_{M/k})^{\mathcal{D}}\sum_{P_{\fN/M}(\ZZ)\in\mathbb{P}_M(X)}C^{|P_{\fN/M}(\ZZ)|}\eeqal
We can split each $P_{\fN/M}(\ZZ)$ into the union of two disjoint sets, $P_{\fN/M}(\ZZ)=U_{\fN/M}(\ZZ)\cup V_{\fN/M}(\ZZ)$ such that

\begin{itemize}
    \item Every  $\ell\in U_{\fN/M}(\ZZ)$ is such that $\ell$ is tamely ramified in $\fN/\QQ$ and $\ell\nmid d_{M/\QQ}.$
    \item Every  $\ell\in V_{\fN/M}(\ZZ)$ is such that $\ell| d_{M/\QQ}[\fN:\QQ]$.
\end{itemize}
Hence
\[\mathbb{P}_M(X)=\left(\bigcup_{\substack{\fN/M\\ [\fN:M]=\fm, \fN/M \text{ is abelian}\\\nm_{M/\QQ}(d_{\fN/M})\leq X}}U_{\fN/M}(\ZZ)\right)\cup\left(\bigcup_{\substack{\fN/M\\ [\fN:M]=\fm, \fN/M \text{ is abelian}\\\nm_{M/\QQ}(d_{\fN/M})\leq X}}V_{\fN/M}(\ZZ)\right).\]

 This implies that
\eqal
\nm_{k/\QQ}(d_{M/k})^{\mathcal{D}}\sum_{P_{\fN/M}(\ZZ)\in\mathbb{P}_M(X)}C^{|P_{\fN/M}(\ZZ)|}\ll 
\nm_{k/\QQ}(d_{M/k})^{\mathcal{D}} \sum_{U_{\fN/M}(\ZZ)\in \mathbb{P}_M(X)}C^{|U_{\fN/M}(\ZZ)|}\sum_{V_{\fN/M}(\ZZ)\in \mathbb{P}_M(X)}C^{|V_{\fN/M}(\ZZ)|}.\eeqal
As $d_{M/\QQ}[\fN:\QQ]$ is fixed,
\eqal\label{power2}
\sum_{V_{\fN/M}(\ZZ)\in \mathbb{P}_M(X)}C^{|V_{\fN/M}(\ZZ)|}\leq \sum_{d|d_{M/\QQ}[\fN:\QQ]}\mu^2(d)C^{\omega(d)}=(C+1)^{\omega(d_{M/\QQ}[N:\QQ])}\ll_{k,G}C_1^{\omega(d_{M/\QQ})}.\eeqal

For a prime $\ell\in U_{\fN/M}(\ZZ)$, denote the exact power of $\ell$ dividing $\nm_{M/\QQ}(d_{\fN/M})$ by $\nu_{\ell}(\nm_{M/\QQ}(d_{\fN/M})).$ Since $\ell$ is tamely ramified, we have that

\eqalnt
\ell\sO_{\fN}=
\prod_{i=1}^{g_\ell(\fN/\QQ)}\sa_{i}^{e(\sa_{i},\ell)}
\qquad\nu_{\ell}(\nm_{M/\QQ}(d_{\fN/M}))=&
\sum_{i=1}^{g_\ell(\fN/\QQ)}f(\sa_i,\ell)(e(\sa_i,\ell)-1)
\eeqal
with $e(\sa_i,\ell)$ and $f(\sa_i,\ell)$ being the respective ramification degrees and inertia degrees. Consider first the case that $\fN/k$ is a Galois extension.
In this case $m=[\fN:M]$ and we see that 
$\nu_{\ell}(\nm_{M/\QQ}(d_{\fN/M}))\geq |G|(1-p^{-1})$ where $p$ is the smallest prime divisor of $m.$ Hence we have

\eqal\label{power3}
\sum_{U_{\fN/M}(\ZZ)\in \mathbb{P}_M(X)}C^{|U_{\fN/M}(\ZZ)|}\ll \sum_{n^{|G|(1-p^{-1})}\leq X}\mu^2(n)C^{\omega(n)}\ll_{k,G}X^{\frac{1}{|G|(1-p^{-1})}+\epsilon}.
\eeqal
If $\fN/k$ is not a Galois extension, then the smallest any $e(\sa_i,\ell)>1$ can be is $p$ where $p$ is the smallest prime divisor of $|F|.$ By combining \eqref{power1}, \eqref{power2} and \eqref{power3}, we get \eqref{power}.


\end{proof}

Note that in the above proof, we only made use of the fact that $\text{Gal}(\hat{\fN}/M)=F$ and were not able to make use of the fact that $\text{Gal}(\hat{\fN}/k)=G.$

\begin{prop}\textnormal{[Abel Summation]}
Let $f$ and $g$ be functions with $f:(\NN\cap [1,X])\rightarrow \mathbb{C}$ and let $g$ be a  differentiable function on $[1,X].$
Let $M_f(X):=\sum_{n\leq X}f(n)$, we have
\eqal
\sum_{n\leq X}f(n)g(n)=M_f(X)g(X)-\int_1^XM_f(t)g^{\prime}(t)dt.
\eeqal
\end{prop}

\subsection{Proof of main theorem}
Throughout this section, we use the terminology established in Notation \ref{H-int} and Diagram \ref{d1}.
We let $G\in\mathcal{F}_1$ be the Galois group of the Galois extension $N/k.$ Our strategy to obtain an upper bound for $N_m(k,G;X)$ for $G\in\mathcal{F}$  and $N_{mt}(k,G;X)$ for $G\in\mathcal{F}_1$
is to make use of \eqref{K/k-rel} and the following discriminant relation
\eqal\label{gen-N/k}
&d_{N/k}=d_{M/k}^m\nm_{M/k}(d_{N/M}).\\
\eeqal
For $G\in\mathcal{F}_1$, using \eqref{gen-N/k} we have

\eqal\label{gal}
N_{mt}(k,G;X)\leq\sum_{\substack{M/k\\ \nm_{k/\QQ}(d_{M/k})\leq  X^{1/m}\\\text{Gal}(M/k)=H}}\sum_{\substack{N/M\\ [N:M]=m, \text{ Gal}(N/M)=F \\  \nm_{M/\QQ}(d_{N/M})\leq X\nm_{k/\QQ}(d_{M/k})^{-m} }}1.\eeqal
Similarly for $G\in\mathcal{F}$, using \eqref{K/k-rel} we have
\eqal\label{gal2}
N_{m}(k,G;X)\leq\sum_{\substack{M/k\\ \nm_{k/\QQ}(d_{M/k})\leq  X^{t/(m-1)}\\\text{Gal}(M/k)=H}}\sum_{\substack{N/M\\ [N:M]=m, \text{ Gal}(N/M)=F \\  \nm_{M/\QQ}(d_{N/M})\leq X^{t}\nm_{k/\QQ}(d_{M/k})^{-(m-1)}} }1.\eeqal

We first bound \eqref{gal2}. By Lemma \ref{powerlem} the inner sum above is 

\eqalnt
\sum_{\substack{ N/M\\ [N:M]=m, \text{ Gal}(N/M)=F \\ \nm_{M/\QQ}(d_{N/M})\leq X^{t}\nm_{k/\QQ}(d_{M/k})^{-(m-1)}} }1\ll X^{\frac{p}{m(p-1)}}\nm_{k/\QQ}(d_{M/k})^{\mathcal{D}-\frac{m-1}{mt(1-p^{-1})}+\epsilon}C^{\omega(\nm_{k/\QQ}(d_{M/k}))}.\eeqal
This implies that 
\eqal\label{gal22}
N_{m}(k,G;X)\ll&X^{\frac{p}{m(p-1)}+\epsilon}\sum_{\substack{M/k\\ \nm_{k/\QQ}(d_{M/k})\leq  X^{t/(m-1)}\\\text{Gal}(M/k)=H}}\nm_{k/\QQ}(d_{M/k})^{\mathcal{D}-\frac{m-1}{mt(1-p^{-1})}+\epsilon}C^{\omega(\nm_{k/\QQ}(d_{M/k}))}.
\eeqal
By assumption we have  \[
\sum_{\substack{M/k\\ \nm_{k/\QQ}(d_{M/k})\leq  Y\\\text{Gal}(M/k)=H}}1=N_t(k,H;Y)\ll Y^{a_1(H,t)+\epsilon}.\] Consequently, by Abel summation, \eqref{gal22} is bounded above by 
\eqal\label{fin1}
N_{m}(k,G;X)\ll&X^{\frac{p}{m(p-1)}+\epsilon}\left(X^{\frac{t}{m-1}\left(a_1(H,t)+ \mathcal{D}-\frac{m-1}{mt(1-p^{-1})}\right)+\epsilon}+O(1)\right)\\
\ll&X^{\frac{p}{m(p-1)}+\epsilon}+X^{\frac{t\left(a_1(H,t)+ \mathcal{D}\right)}{m-1}+\epsilon}.\eeqal
Using \eqref{gal} we now obtain the upper bound for $N_{mt}(k,G;X)$ in similar fashion.
\eqal\label{fin}
N_{mt}(k,G;X)\ll&X^{\frac{p}{mt(p-1)}+\epsilon}
\sum_{\substack{M/k\\ \nm_{k/\QQ}(d_{M/k})\leq  X^{t/(m-1)}\\\text{Gal}(M/k)=H}}\nm_{k/\QQ}(d_{M/k})^{\mathcal{D}-\frac{p}{t(p-1)}+\epsilon}C^{\omega(\nm_{k/\QQ}(d_{M/k}))}\\
\ll&X^{\frac{p}{mt(p-1)}+\epsilon}\left(X^{\frac{1}{m}\left(a_1(H,t)+\mathcal{D}-\frac{p}{t(p-1)}\right)}+O(1)\right)\\
\ll&X^{\frac{p}{mt(p-1)}+\epsilon}+X^{\frac{a_1(H,t)+\mathcal{D}}{m}+\epsilon}.\\
\eeqal
This proves Theorem \ref{main-2}. 

We now prove Corollary \ref{main-1}. If Malle's conjecture holds for $N_t(k,H;X)$, then $a_1(H,t)=p_1/(t(p_1-1))$. Using this and assuming the $\ell$-torsion conjecture, \eqref{fin} implies that 
\[N_{mt}(k,G;X)\ll X^{\frac{p}{mt(p-1)}+\epsilon}+X^{\frac{p_1}{mt(p_1-1)}+\epsilon}.\]
Using the above and equation \eqref{a(G)=mt}
shows part $(1)$ of Corollary \ref{main-1}.
Assuming the $\ell$-torsion conjecture and Malle's conjecture for $N_t(k,H;X)$, \eqref{fin1} agrees with \eqref{a(G)=m}. This shows part $(2)$ of Corollary \ref{main-1}.

\section{Explicit Computations and Extensions}

In this section, the tools needed to prove the examples in Example \ref{exmp1} will be addressed. The section has two parts. The first part addresses the flexibility of the method of proof as we may use it to count subfields of various degrees. The next part addresses the occasions in which we can say something non-trivial about the size of the $m$-torsion of the class group. Throughout this section, we use the same notation as in Notation \ref{H-int}. 

\subsection{Extensions}\label{exts}

In this section we obtain upper bounds for $N_6(k,C_3^2\rtimes C_4;X)$ and  $N_6(k,A_4;X).$ In terms of MAGMA notation, we obtain upper bounds for $N_6(k,6T10;X)$ and $N_6(k,6T4;X).$ The method used to obtain these bounds is an extension of the method used to prove Theorem \ref{main-2}, and can be used to bound $N_d(k,G;X)$ for various $d$ that depend on $G.$

\begin{prop}
For any number field $k$, the number of degree 6 extensions $N_1/k$ with $\text{Gal}(\hat{N}_1/k)=A_4$ satisfies
\eqal
N_6(k,A_4;X)\ll X^{1/2+\epsilon}.\eeqal
The upper bound is as predicted by Malle.
\end{prop}
Note that $A_4$ has normal abelian subgroup $F=\{e, (12)(34),(13)(24),(14)(23)\}.$
The subgroup $F_1=\{e,(12)(34)\}$ of $F$ is not normal in $A_4$.
To count the number of degree 6 extensions with Galois group $A_4$, we fix a base field $k$ and let $N/k$ vary over Galois extensions with Galois group $A_4.$  
For each fixed $N/k$, we have the following field diagram with $M=\mbox{Fix}(F)$, and up to conjugacy we set $K=\mbox{Fix}(C_3)$ and $N_1=\mbox{Fix}(F_1).$

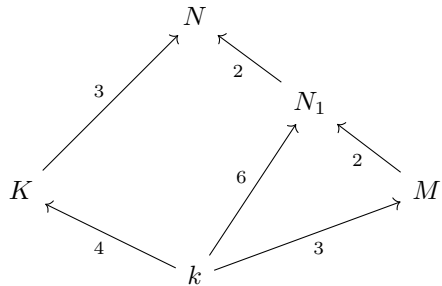
\[
  \begin{tikzcd}
  &  &N   &  & \\ 
  &  &   & N_1\arrow{ul}{2}& \\ 
  K\arrow{uurr}{3}  &  &   &  & M   \arrow{ul}{2}\\
  &  &  k\arrow{ull}{4} \arrow[swap]{urr}{3}\arrow{uur}{6} &  & \\   
  \end{tikzcd}
\]
To show the proposition note that $d_{N_1/k}=d_{M/k}^2\nm_{M/k}(d_{N_1/M})$. Hence, $N_6(k,A_4;X)$ is bounded above by

\eqalnt
\sum_{\substack{M/k\\ \nm_{k/\QQ}(d_{M/k})\leq  X^{1/2}\\\text{Gal}(M/k)=C_3}}\sum_{\substack{N_1/M\\ [N_1:M]=2\\  \text{Gal}(\hat{N_1}/k)=A_4\\  \nm_{M/\QQ}(d_{N_1/M})\leq X\nm_{k/\QQ}(d_{M/k})^{-2}} }1\\
\eeqal
Now we use a result by Kl{\"u}ners. 
\begin{lem}\textnormal{[Kl{\"u}ners, \cite{klu12}, Lemma 4]}\label{kl2}
Let $N_1/M/k$ be extensions of number fields with $\text{Gal}(M/k) = H$ and $[N_1 : M] = 2$. Assume there exists a prime $p$ which is unramified in $M/k$ with   $p||\nm_{M/\QQ}(d_{N_1/M})$. Then $\text{Gal}(\hat{N_1}/k) = C_2 \wr H$.
\end{lem}
Since $A_4\neq C_2\wr C_3$, there can be no prime $p\in\NN$ that is unramified in $M$ and $p|| \nm_{M/\QQ}(d_{N_1/M}).$ This implies that $R$ in Lemma \ref{powerlem}, is 2. Hence by Lemma \ref{powerlem}, 

\eqal
N_{6}(k,A_4;X)
&\ll X^{1/2+\epsilon}
\sum_{\substack{M/k\\ \nm_{k/\QQ}(d_{M/k})\leq  X^{1/2}\\\text{Gal}(M/k)=C_3}}\nm_{k/\QQ}(d_{M/k})^{\mathcal{D}-1+\epsilon}C^{\omega(\nm_{k/\QQ}(d_{M/k}))}\eeqal
Now, using that $M/k$ is a normal extension, and $N_3(k,C_3;X)\ll X^{1/2+\epsilon}$, and using partial summation, we have that
\eqalnt
N_{6}(k,A_4;X)&\ll X^{1/2+\epsilon}\left((X^{1/2})^{1/2+\mathcal{D}-1+\epsilon}+O(1)\right)\ll X^{1/2+\epsilon}.\\\eeqal
This proves the proposition.

\begin{rem}
By the same method we can show results like $N_{14}(k,C_2^3\rtimes C_7;X)\ll X^{1/2+\epsilon}.$

\end{rem}

\begin{exmp}
In this example we show $N_{6}(k,C_3^2\rtimes C_4;X)\ll X^{1/2+\epsilon}$. This improves on the result of Alberts (\cite{ED} Appendix A) who was able to show  $N_6(k,C_3^2\rtimes C_4;X)\ll X^{2+\epsilon}.$ Malle's conjectured upper bound for  $N_6(k,C_3^2\rtimes C_4;X)$ is $O(X^{1/2+\epsilon}),$ as well.

We start by fixing $k$ and letting $N/k$ vary over Galois extensions with Galois group $\text{Gal}(N/k)=C_3^2\rtimes C_4$. As $N/k$ varies, we have the following field diagram where the degree 6 extensions of $k$ are denoted by $N_1/k.$ Here, $N_1/k$ need not be unique. Here $M=\mbox{Fix}(C_3^2)$, $M_1$ is the degree 2 subfield of $M/k$ and up to conjugacy we set $K=\mbox{Fix}(C_4)$.

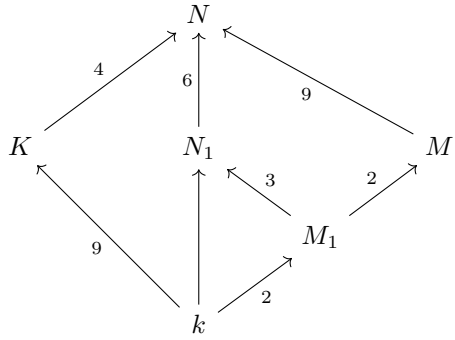
\[
  \begin{tikzcd}
  &  &N   &  & \\ 
  &  &   & & \\ 
  K\arrow{uurr}{4}  &  & N_1\arrow{uu}{6}  &  & M   \arrow{uull}{9}\\  
  &  &   & M_1\arrow{ur}{2}\arrow[swap]{ul}{3}& \\   
  &   & k\arrow{uull}{9} \arrow[swap]{ur}{2}\arrow[swap]{uu} &  &
  \end{tikzcd}
\]
We note that $d_{N_1/k}=d_{M_1/k}^3\nm_{M_1/k}(d_{N_1/M_1})$ and that $M/k$ is a Galois extension. Since $C_2$ is normal in $C_4$, $M_1/k$ is a Galois extension. Fixing $M_1/k$ we count the extensions $N_1/M_1$ with discriminant supported on a certain modulus. In particular, $N_6(k,C_3^2\rtimes C_4;X)$ is bounded above by

\eqalnt
\sum_{\substack{M_1/k\\ \nm_{k/\QQ}(d_{M_1/k})\leq  X^{1/3}\\\text{Gal}(M_1/k)=C_2}}\sum_{\substack{N_1/M_1\\ [N_1:M_1]=3\text{, } N_1/M_1\text{ is abelian}\\  \text{Gal}(\hat{N_1}/k)=C_3^2\rtimes C_4\\  \nm_{M/\QQ}(d_{N_1/M_1})\leq X\nm_{k/\QQ}(d_{M_1/k})^{-3}} }1.\\
\eeqal
By Lemma \ref{powerlem} for $N_1/M_1$, we have that $R=2$ and we notice that
\eqalnt
N_{6}(k,C_3^2\rtimes C_4;X)&\ll X^{1/2+\epsilon} \sum_{\substack{M_1/k\\ \nm_{k/\QQ}(d_{M_1/k})\leq  X^{1/3}\\\text{Gal}(M_1/k)=C_2}} \nm_{k/\QQ}(d_{M_1/k})^{\mathcal{D}-3/2}C^{\omega( \nm_{k/\QQ}(d_{M_1/k}))}\\
&\ll X^{1/2+\epsilon}\left((X^{1/3})^{1+\mathcal{D}-3/2+\epsilon}+O(1)\right)\ll X^{1/2+\epsilon}\\
\eeqal
where $\mathcal{D}=\mathcal{D}(C_2,3)$.

Above we see that since $1+\mathcal{D}-3/2\leq0,$ getting a better upper bound than $1/2$ for $\mathcal{D}$ will not further improve the upper bound for $N_6(k,C_3^2\rtimes C_4;X).$ 

\begin{rem}
The same method can be used to show that $N_6(k,C_3^2\rtimes C_2;X)\ll X^{1/2+\epsilon}$ and $N_6(k,D_{6};X)\ll X^{1/2+\epsilon}$, these are exactly what Malle's conjecture predicts. 
\end{rem}

\end{exmp}


\subsection{Size of the Class group}
The best general upper bound for $\mathcal{D}$ is $1/2+\epsilon$ however, we can do better in certain cases. For instance, in Example \ref{exmp1} the result $N_4(k,A_4;X)\ll X^{0.77+\epsilon}$ makes use of non-trivial bounds for the two torsion of the class group. In particular,  Bhargava, Shankar,  Taniguchi, Thorne,  Tsimerman, and Zhao \cite{2tors} show that for any degree 3 extension $M/\QQ$, $|\text{Cl}_M[2]|\ll_{\epsilon}|d_{M/\QQ}^{0.2784\dots+\epsilon}|.$
Similarly, we explore the consequence of another such result about non-trivial bounds on the size of the torsion of the class group. If $M/\QQ$ is a Galois extension with Galois group $C_{p_1}$ where $p_1$ is a prime, then we can use of non trivial estimates for $\text{Cl}_M[\ell]$ for any $\ell.$ We make use of this in Example \ref{exmp1} where we show that $N_{103}(\QQ,C_{103}\rtimes C_{17};X)\ll X^{0.09369+\epsilon}$ (as opposed to what we would get otherwise, ie, $\ll X^{0.09375+\epsilon}$). Pierce, Turnage-Butterbaugh and Wood establish non-trivial bounds on $\mathcal{D}$ for most prime cyclic extensions $M/\QQ$, and this suffices for our purpose. In order to make this rigorous, we need to establish terminology as stated in \cite{PTBW}. 

\begin{defn}{($\delta-$exceptional field)}

Let $M/\QQ$ be a degree $n$ extension with Galois group $G.$ Such an extension is called a $\delta-$exceptional field for $0<\delta<1/2$ precisely when the Dedekind zeta
function of the Galois closure $\hat{M}$ of $M$ over $\QQ$ has the property that 
$\zeta_{\hat{M}}(s)/\zeta(s)$ has a zero in the region 
\[ [1-\delta]\times[-(\log{d_{\hat{M}/\QQ}})^{2/\delta},(\log{d_{\hat{M}/\QQ}})^{2/\delta}].\]
\end{defn}

\begin{thm}\textnormal{[\cite{PTBW}, Theorem 1.19]}\label{ptw}
Let $M/\QQ$ represent degree $p$ extensions with Galois group $C_p$ where $p$ is a prime, with $d_{M/\QQ}\leq X$. Fix $0<\epsilon_0<1/(4(p-1))$. Define 
\[\delta=\frac{\epsilon_0}{5p+2/(p-1)+4\epsilon_0}.\]
Then we have that there are at most 
$O_{p,\epsilon_0}(X^{\epsilon_0})$ fields $M/\QQ$ that are $\delta-$exceptional. Aside from the $\delta-$exceptional fields, every field extension $M/\QQ$ that is counted  in $N_p(\QQ,C_p;X)$ satisfies the following, for every $\ell\in\NN$

\eqal
|\text{Cl}_{M}[\ell]|\ll_{p,\ell,\epsilon}d_{M/\QQ}^{\frac{1}{2}-\frac{1}{2\ell(p-1)}+\epsilon}.
\eeqal
\end{thm}

In fact, Pierce, Turnage-Butterbaugh and Wood have shown similar statements for a larger set of groups. We can make use of their results to obtain better upper bounds for certain families of groups, in particular, those families where $H$ is a group of the form $C_{p_1}$ where $p_1$ is a prime. 
We now make precise how to make use of their results. Let $G=C_m\rtimes C_{p_1}\in\mathcal{F}$, we compute $N_{m}(\QQ,G;X),$ and as in Diagram \ref{d1}
$M=\text{Fix}(F).$
Let $D_{\epsilon_0}(X)$ be the set of $\delta$-exceptional fields $M/\QQ$ and $|d_{M/\QQ}|\leq X$. 
Let $w_A=|\{M/\QQ: \nm(d_{M/\QQ})=A\}|$. We may write equation \eqref{gal22} as

\eqalnt
&N_m(\QQ,C_m\rtimes C_{p_1};X)\ll
X^{1/m(1-p^{-1})+\epsilon}\sum_{A\leq X^{t/(m-1)}}w_AA^{\mathcal{D}-\frac{m-1}{mt(1-p^{-1})}+\epsilon}\\
&\ll X^{1/m(1-p^{-1})+\epsilon}
\left(\sum_{\substack{A\leq X^{t/(m-1)} \\ M\notin D_{\epsilon_0}(X^{t/(m-1)})}}w_AA^{\mathcal{D}_1-\frac{m-1}{mt(1-p^{-1})}+\epsilon}+\sum_{\substack{A\leq X^{t/(m-1)} \\ M\in D_{\epsilon_0}(X^{t/(m-1)})}}w_AA^{\mathcal{D}_2-\frac{m-1}{mt(1-p^{-1})}+\epsilon}\right).\\
\eeqal
Here we have $\mathcal{D}_1=1/2(1-1/(p_1m-m))$ and $\mathcal{D}_2=1/2$. By Theorem \ref{ptw}  \[\sum_{\substack{A\leq X\\ M\in D_{\epsilon_0}(X)}}w_A=O(X^{\epsilon_0}).\] Consequently the above is 
\eqalnt
&\ll X^{1/m(1-p^{-1})+\epsilon}
\left(X^{\frac{t}{m-1}\left(\frac{1}{p_1-1}+\frac{1}{2}-\frac{1}{2m(p_1-1)}-\frac{m-1}{mt(1-p^{-1})}\right)+\epsilon}+X^{\frac{t}{m-1}\left(\epsilon_0+\frac{1}{2}-\frac{m-1}{mt(1-p^{-1})}\right)+\epsilon}+O(1)\right).\\
\eeqal
The second term in the parenthesis will not contribute to the main term as we can take $\epsilon_0<1/(4p_1-4)$. This implies that 
\eqal
&N_m(\QQ,C_m\rtimes C_{p_1};X)\ll  X^{1/m(1-p^{-1})+\epsilon}+X^{\frac{t}{m-1}\left(\frac{1}{p_1-1}+\frac{1}{2}-\frac{1}{2m(p_1-1)}\right)+\epsilon}.\eeqal
Using that above, we can show that $N_{103}(\QQ,C_{103}\rtimes C_{17};X)\ll X^{0.09369+\epsilon}.$
We can similarly apply this technique to $N_{mt}(\QQ,G;X)$ for $G\in\mathcal{F}_1.$
\section{Final remarks}

Finally we describe the limitations  of this method. We will use the same notation here as that in Notation \ref{H-int}. Let $N/N_1/M_1/k$ be a tower of number fields with the following conditions:

1. $\text{Gal}(N/k)=\text{Gal}(\hat{N_1}/k)=G\in\mathcal{F}_1$. 

2. $M$ is the fixed field of $F$,  and  $M_1$ is any subfield of $M/k$. Let number of $M_1/k$ of degree $[M_1:k]$ with $\nm_{k/\QQ}(d_{M_1/k})\leq X$ be bounded above by $X^{a(H,[M_1:k])}.$

We classify when we need to address $|\text{Cl}_{M_1}[[N_1:M_1]]|$ and when we do not have to, when finding an upper bound for $N_{[N_1:k]}(k,G;X).$ We have the following corresponding field diagram.
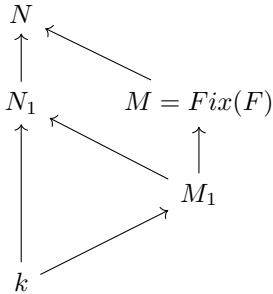
\[
  \begin{tikzcd}
    N   &   \\  
  N_1\arrow{u}    &     M=Fix(F)   \arrow{ul}\\  
       & M_1\arrow{u}\arrow[swap]{ul}  \\   
     k\arrow{uu} \arrow[swap]{ur} &  &
  \end{tikzcd}
\]
We have $d_{N_1/k}=d_{M_1/k}^{[N_1:M_1]}\nm_{M/\QQ}(d_{N_1/M_1}).$ We go over the same process as in equation \eqref{fin}. 
Let $R$ be the smallest exponent of a prime $\ell\in\NN$ that is unramified in $M_1$ and ramified tamely in $N_1/\QQ$.
Recall by Lemma \ref{powerlem}, $R=p-1$ if $N_1/M_1$ is not a Galois extension, and $R=[N_1:k](1-p^{-1})$ in the case that $N_1/k$ is a Galois extension. If 
\eqal\label{ineq1}
a(H,[M_1:k])+|\text{Cl}_{M_1}[[N_1:M_1]]|-\frac{[N_1:M_1]}{R}\leq0
\eeqal
then $N_{[N_1:k]}(k,G;X)\ll X^{1/R+\epsilon}.$ 
In the case that \eqref{ineq1} does not hold, we have, 
$N_{[N_1:k]}(k,G;X)\ll X^{r+\epsilon}$ where 
\[r=\frac{a(H,[M_1:k])+|\text{Cl}_{M_1}[[N_1:M_1]]|}{[N_1:M_1]}.\]
In the case that \eqref{ineq1} does hold, improving the upper bound for $|\text{Cl}_{M_1}[[N_1:M_1]]|$ does not play a further role in improving $N_{[N_1:k]}(k,G;X).$ 
The only obstacle from reaching Malle's conjectured upper bounds in these cases is being able to increase the value of the constant $R$ (as done by Kl{\"u}ners in Lemma 4 of \cite{klu12}). As Lemma \ref{powerlem} indicates, the inequality \eqref{ineq1} is less likely to hold in the case that $N_1/k$ is a Galois extension. In fact, when $N_1/k$ is Galois, $N_1=N$ and $M_1=M$, and if Malle's conjecture is realized for $N_{[M:k]}(k,H;X),$ then \eqref{ineq1} holds when

\[\frac{1}{[M:k]}\left(\frac{p_1}{p_1-1}-\frac{p}{p-1}\right)+|\text{Cl}_{M}[[N:M]]|\leq0.\]
Here $p$ and $p_1$ above are the smallest prime divisors of $[M:k]$ and $[N:M]$ respectively. Under the assumption of the $\ell$-torsion conjecture, the above holds only if $p<p_1.$ However, there may be instances where results weaker than the $\ell$-torsion conjecture may suffice. For instance when computing $N_8(k,F_{56};X)$ where showing $\text{Cl}_M[2]\ll d_{M/\QQ}^{1/12}$ suffices to show an upper bound that matches Malle's conjecture.

\section{Acknowledgements}
I would like to thank my advisor Frank Thorne, and Alex Duncan for many insightful and encouraging discussions.

\InputIfFileExists{Frobenius-2.bbl}

\end{document}